\documentclass[11pt,reqno]{amsart}
\newtheorem{theorem}{Theorem}[section]

\newtheorem{corollary}[theorem]{Corollary}
\newtheorem{remark}[theorem]{Remark}
\newtheorem{lemma}[theorem]{Lemma}
\newtheorem{example}[theorem]{Example}
\newtheorem{definition}[theorem]{Definition}
\newtheorem{conjecture}[theorem]{Conjecture}

\usepackage{amssymb}
\usepackage{graphicx}
\usepackage{amsmath}

\usepackage[all]{xy}
\usepackage[latin1]{inputenc} 
\usepackage{graphicx} 
\numberwithin{equation}{section}

\begin{document}
\title[]{Linear stability and a stability of Lazarsfeld-Mukai bundles}
\author{Abel Castorena and H. Torres-L\'opez}

\address{Centro de Ciencias Matem\'aticas. Universidad Nacional Aut\'onoma de M\'exico, Campus Morelia.}
\email{abel@matmor.unam.mx}
\email{hugo@matmor.unam.mx}
\thanks{The first named author was supported by Research Grant PAPIIT IN100716 (UNAM).\\ 
The second named author is supported with a Posdoctoral Fellowship from CONACyT, and was partially supported by Research Grant PAPIIT IN100716 (UNAM)}




\date{\today}

\begin{abstract}Let $C$ be a smooth irreducible projective curve and let $(L,H^0(C,L))$ be a complete and generated linear series on $C$. Denote by $M_L$ the kernel of the evaluation map $H^0(C,L)\otimes\mathcal O_C\to L$. The exact sequence $0\to M_L\to H^0(C,L)\otimes\mathcal O_C\to L\to  0$ fits into a commutative diagram that we call the Butler's diagram. This diagram induces in a natural way a multiplication map on global sections $m_W: W^{\vee}\otimes H^0(K_C)\to H^0(S^{\vee}\otimes K_C)$, where $W\subseteq H^0(C,L)$ is a subspace and $S^{\vee}$ is the dual of a subbundle $S\subset M_L$. When the subbundle $S$ is a stable bundle, we show that the map $m_W$ is surjective. When $C$ is a Brill-Noether general curve, we use the surjectivity of $m_W$ to give another proof on the semistability of $M_L$, moreover we fill up a gap of an incomplete argument by Butler: With the surjectivity of $m_W$ we give conditions to determinate the stability of $M_L$, and such conditions implies the well known stability conditions for $M_L$ stated precisely by Butler. Finally we obtain the equivalence between the stability of $M_L$ and the linear stability of $(L,H^0(L))$ on $\gamma$-gonal curves.
 \end{abstract}

\maketitle

\section{Introduction}\label{intro}

Let $C$ be a smooth irreducible projective curve of genus $g$ over an algebraically closed field $k$ of characteristic zero, and let $K_C$ be the canonical bundle on $C$. A generated linear series of type $(d,r+1)$ over $C$ is a pair $(L,V)$, where $L$ is a generated line bundle of degree $d$ on $C$ and $V\subseteq H^0(C,L)$ is a linear subspace of dimension $r+1$ that generates $L$. The kernel $M_{V,L}$ of the evaluation map $V\otimes \mathcal{O}_C\rightarrow L$ fits into the following exact sequence	
\begin{equation}\label{dualspam}
 \xymatrix{ 0 \ar[r]^{} &  M_{V,L} \ar[r]^{}&  V\otimes \mathcal{O}_C\rightarrow  \ar[r]^{} &  L \ar[r]^{} & 0.}
\end{equation}

\noindent The bundle $M_{V,L}$ is called  Lazarsfeld-Mukai bundle. When $V=H^0(C,L)$, we will denote the bundle $M_{H^0(L),L}$ by $M_L$. The vector bundle $M_{V,L}$ and its dual $M_{V,L}^{\vee}$ have been studied from different points of view because of the rich geometry they encode. The study of the stability of $M_{V,L}$ is related with: the study of Brill-Noether varieties (see \cite{bbn1}), the Resolution Minimal Conjecture (see \cite{resolucionminimal}), the stability of the tangent bundle of a projective space restricted to a curve; and the theta divisors of vector bundles on curves (see \cite{mis}, \cite{mistretastopino}). Ein and Lazarsfeld used the stability of $M_{V,L}$ to prove the stability of the Picard bundle (see \cite{ein}). In (\cite{ram}), Paranjape and Ramanan proved that $M_{K_C}$ is semistable, and David C. Butler showed that $M_L$ is stable for $d>2g$, and it is semistable for $d=2g$ (see \cite{butler} and \cite{ram}). 
\vspace{.25cm}

\noindent David Mumford introduced the concept of linear stability for projective varieties $X\subset\mathbb{P}^n$ (see \cite{mumfordstability}). In some sense, this definition is a way to measure how $X$ sits in $\mathbb{P}^n$. It was generalized for linear series $(L,V)$ over a curve $C$ (see \cite{mistretastopino}). Linear stability of a generated linear series $(L,V)$ is a weaker condition than the stability for the vector bundle $M_{V,L}$, that is, the stability of $M_{V,L}$ implies the linear stability of the pair $(L,V)$. In this direction, E. C. Mistretta showed that the linear stability of $ (L,V)$ is equivalent to the stability of $ M_{V,L}$ when $d\geq 2g+2c$ and $V\subseteq H^0(L)$ is a subspace of codimension $c\leq g$ (see \cite{mis}, Lemma 2.2). Using this equivalence he showed that for a general subspace $V\subseteq H^0(L)$ of codimension $c\leq g$,  $M_{V,L}$ is semistable (see \cite{mis}, Theorems 2.7 and 2.8).
  
 \vspace{.25cm}
 
\noindent E. C. Mistretta and L. Stoppino gave conditions for the equivalence between the stability of $M_{V,L}$ and linear stability of a linear series  $(L,V)$ of type $(d,r+1)$, when  $d\leq 2r+\text{Cliff}(C)$ (see \cite{mistretastopino}, Theorem 1.1). With this equivalence they proved the stability of $M_{L}$ when $L$ computes the Clifford index or $d\geq 2g-\text{Cliff}(C)$. The importance in giving this equivalence is because it is easier to prove the linear stability for a pair $(L,V)$ instead of the stability of $M_{V,L}$.

 \vspace{.25cm}
 
In this paper we are interested in two goals. The first one, is to study the multiplication map of global sections for certain vector bundles appearing in a natural way when we dualize a diagram (see diagram \ref{butlerdiagram} below) induced by the choice of $(L,V)$. The second one, is to give a positive answer to a conjecture of C. Mistretta and L. Stoppino for the case of $\gamma$-gonal curves (see \cite{mistretastopino} Conjecture 8.7). In this direction, our goal is to give conditions under which the stability of $M_L$ is equivalent to the linear stability of $(L, H^0(L))$ (see Corollary \ref{equivalenciageneral}). Moreover, we give a criterium where this equivalence is satisfied on $\gamma$-gonal curves(see Theorem \ref{gonalinyectivo}, section 5). To explain in our context this multiplication map and our results we recall the following: 
\vskip2mm

\noindent Let $L$ be a generated line bundle $L$ over a curve $C$ and let $V\subseteq H^0(C,L)$ be a subspace of sections that generates $L$, consider the kernel of the evaluation map $M_{V,L}=\text{ker}(V\otimes O_C\to L)$. We recall that given a subbundle $S\subseteq M_{V,L}$, there exists a subspace $W^{\vee}\subseteq H^0(C,S^{\vee})$ and a bundle $F_S$ that fit into the following diagram:


\begin{equation}\label{butlerdiagram}
 \xymatrix{0  \ar[r]  & S  \ar[r] \ar@{^{(}->}[d]      &  W\otimes\mathcal{O}_C \ar[r] \ar@{^{(}->}[d] & F_S \ar[r] \ar@{->}[d]^{\alpha}  & 0\\
				0\ar[r]^{}	& M_{V,L} \ar[r]_{}& V\otimes\mathcal{O}_C  \ar[r]_{}   & L \ar[r]^{} & 0    .\\  }
\end{equation}

Indeed, we define $ W \hookrightarrow V$ by $W^{\vee}:=\text{Im}(V^{\vee}\stackrel{\phi}{\rightarrow} H^{0}(C,S^{\vee}))$. Note that $W^{\vee}$ generates $S^{\vee}$ because $W$ is a subspace of $V$.  Thus,
we define $F_S^{\vee}:=\text{Ker}(W^{\vee}\otimes\mathcal{O}\rightarrow  S^{\vee}).$ We call the diagram \ref{butlerdiagram} the {\bf Butler's diagram of }$(L,V)$ by $S$. When $V=H^0(C,L)$, we just call the diagram \ref{butlerdiagram} the {\bf Butler's diagram of} $L$ by $S$. 

\vskip2mm

\noindent When we dualize the first exact row in the diagram (\ref{butlerdiagram}) and we twist by $K_C$, we take cohomology to obtain the following multiplication map of sections
 \begin{equation*}
  \xymatrix{m_W: W^{\vee}\otimes H^0(C,K_C)\ar[r]^{} & H^0(C,S^{\vee}\otimes K_C).}
 \end{equation*}

\noindent  The rank of multiplication map of sections for vector bundles on curves has interesting geometric meanings and appears in many contexts in algebraic geometry. In this direction we consider the complete case $V=H^0(C,L)$ and the corresponding Butler's diagram of $L$ by $S$. We have the following results.

	\begin{theorem}\label{intromultiplicationsections}
	 Let $Q:=M_L/S$ be the quotient of $M_L$ by $S$. 
	 \begin{enumerate}
	 \item   The multiplication map 
	 \begin{equation*}
  \xymatrix{m_W: W^{\vee}\otimes H^0(C,K_C)\ar[r]^{} & H^0(C,S^{\vee}\otimes K_C),}
 \end{equation*}
	 
is surjective if and only if  $H^0(Q)=0$.

\vspace{.2cm}

\item  If $m_W$ is surjective, then $W=H^0(F_S)$.

\vspace{.2cm}

\item  If $S\subset M_L$ is stable of maximal slope, then $H^0(Q)=0$.

\vspace{.2cm}
\end{enumerate}
	\end{theorem}

\noindent In the above theorem, the proof of condition (3) is valid also for a non-complete linear series $(L,V)$. 

\noindent When $C$ is a general curve in the sense of Brill-Noether, we apply the surjectivity of $m_W$ to give a different proof for the semistability of $M_L$ (see \cite{bra}, \cite{schneider}). Moreover, with the surjectivity of $m_W$ we obtain conditions to determinate the stability of $M_L$, these conditions are summarized in Lemma \ref{semistabilityML}, Section 4 (see also \cite{bra}, Lemma 3.2). With these conditions we fill up a gap in the proof of (\cite{but}, Theorem 2.2). We think that the surjectivity of $m_W$ in Theorem 1.1 is a different approach in the study of the (semi)stability of $M_L$. This is the spirit of the following result (see Theorem \ref{generalcomplete}, Section 4):

 \begin{theorem}\label{introcomplete}
	Let $C$ be a general curve of genus $g\geq 2$, and let $L\in \text{Pic}^d(C)$ be a  globally generated line bundle on $C$ of degree $d$. Consider $S\subset M_L$ a stable subbundle with $\mu(S)=\mu(M_L)$, then there exists a line bundle $F=F_S$ that fits into the following commutative diagram
\begin{equation*}	
 \xymatrix{0  \ar[r]  & S  \ar[r] \ar@{^{(}->}[d]      &  H^0(C,F)\otimes\mathcal{O}_C  \ar[r] \ar@{^{(}->}[d] & F_S \ar[r] \ar@{->}[d]^{\alpha}  & 0\\
				0\ar[r]^{}	& M_{L} \ar[r]_{}& H^0(C,L)\otimes\mathcal{O}_C  \ar[r]_{}   & L \ar[r]^{} & 0    .\\  }
				\end{equation*}
\end{theorem}

\noindent When the curve $C$ is Brill-Noether general, Theorem \ref{introcomplete} provides a natural and intrinsic characterization of Butler's diagrams of $L$ by $S$ with $S$ of maximal slope. Moreover, the advantage of the rank of $F_S$ being one lies in the fact  that linear stability is equivalent to stability of $M_L$. It is worth pointing out that Theorem \ref{introcomplete} gives precise conditions for the stability of $M_L$. These conditions were stated in (\cite{but}, Theorem 2.2) and we add them in the following corollary (see Corollary \ref{equivalenciageneral}, Section 4).
 
\begin{corollary}\label{introequivalence}
 Let $L$ be a globally generated line bundle over a general curve $C$ of genus $g$. Suppose that $h^0(C, L)=r+1$, then 
 \begin{enumerate}
 \item Linear (semi)stability of $L$ is equivalent to (semi)stability of $M_L$.
 \item $M_L$ fails to be stable if and only if all the following three conditions hold
	 \begin{enumerate}
	  \item $h^1(C,L)=0$.
	  \item  $\text{deg}(L)=g+r$ and $r$ divides  $g$.
	  \item There is an effective divisor $Z$ with $h^0(C,L(-Z))=h^0(C,L)-1$ and $\text{deg}(Z)=1+\frac{g}{r}$.
	    \end{enumerate}
	  \end{enumerate}
	  \end{corollary}
  
\noindent To prove Corollary 1.3 we apply Theorems \ref{intromultiplicationsections} and \ref{introcomplete}. The proof of Corollary 1.3 shows the importance of the surjectivity of the multiplication map $m_W$ to obtain the above condition (c). Also, this condition is not included in (\cite{bra}, Proposition 3.5)

\vskip2mm

\noindent Finally we give the following criterium on $\gamma$-gonal curves to state the equivalence between linearly stable and stability of $M_L$ (see Theorem \ref{gonalinyectivo}):

 \begin{theorem}\label{introgonal}
	Let $L\in \text{Pic}^d(C)$ be a globally generated line bundle over a curve $C$ of gonality $\gamma$ and suppose that $h^0(C,L)=r+1$. Consider a line bundle $B$ on $C$ such that $|B|=g^1_{\gamma}$, and suppose that $d>\gamma\cdot  r$. Then, if the multiplication map 
 \begin{equation*}
\xymatrix{   \mu_{L}^{\gamma}:=H^0(C,L)\otimes H^0(C,B) \ar[r]^{} & H^0(C,L\otimes B)}
 \end{equation*}

 is not injective, then $L$ is not linearly semistable.
	 \end{theorem}

This article is organized as follows. In Section 2, we review some standard facts on the Butler's diagram and the linear stability for a linear series $(L,V)$. In Section 3 we study properties of the Butler's diagram to prove Theorem \ref{intromultiplicationsections}. We consider the case when $C$ is a Brill-Noether general curve and we prove Theorem \ref{introcomplete} (see Theorem \ref{generalcomplete}, section 4) and  Corollary \ref{introequivalence} (see Corollary \ref{equivalenciageneral}, section 4). Finally in Section 5 we give a criterium between the equivalence of linear stability of $(L,H^0(L))$ and the stability of $M_L$ on $\gamma$-gonal curves (see Theorem \ref{gonalinyectivo}).

\section{preliminaries and notation}
  
 Let $C$ be  a smooth projective and irreducible curve of genus $g\geq 2$ defined over an algebraically closed field $k$ of characteristic zero. Given a vector bundle $E$ over $C$ we denote by $d_E$ (or $\text{deg}(E))$ the degree of $E$, and by $n_E$ the rank (or $\text{rk}(E)$) of 
$E$. The slope of $E$ is defined as the rational number $\mu(E):=\frac{d_E}{n_E}$. For abbreviation, we write $H^i(E)$ instead of $H^i(C,E)$, whenever it is convenient. As usual, $h^i(E)=\text{dim}_k H^i(E)$.

\begin{remark}\label{desestabilizante}
	\begin{em}
		Consider a Butler's diagram of $(L,V)$ by $S$ as in (\ref{butlerdiagram}). The following properties hold (see \cite{butler}).
		
		\begin{enumerate}
	
          	\item $W$ is a  subspace of $H^0(F_S)$.
		
		\vspace{.05cm}
		
			\item The sheaf $F_S$  is generated and $h^0(F_S^{\vee})=0.$
		
		\vspace{.05cm}
		
			\item The induced morphism $\alpha: F_S\rightarrow L$ is not zero.
			
			\vspace{.05cm}

			\item Since $S^{\vee}$ is a generated subbundle with no trivial summands, it follows that the rank of $S$ is smaller than the dimension of $W$.
			
			\vspace{.05cm}
			
			\item Let $S\subset M_{V,L}$ be a subbundle of maximal slope, then $\text{deg}(F_S)\leq\text{deg}(I)$, where $I$ is $\text{Im}(\alpha)$.
			Moreover, if $S$ is the destabilizing bundle of $M_{V,L}$, then  $\text{rk}(F_S)=1$ if and only if $\text{deg}(F_S)=\text{deg}(I)$.

		\end{enumerate}
	\end{em}
\end{remark}

\noindent In the Butler's diagram (1.2), we denote by $w$ the dimension of $W$, by $s$ the rank of $S$ and by $f$ the degree of $F_S$.  
\vskip2mm
 
\noindent We recall that  a vector bundle $E$ is stable (semistable) if for all non-trivial subbundle $F\subset E $
\begin{eqnarray*}
 \mu(F)<\mu(E) \hspace{.6cm} (\text{resp.} \leq).
\end{eqnarray*}
If $E$ is not semistable, then we say that $E$ is unstable. To understand unstables bundles, there exists the invariant, $\mu^+(E) $, which is defined as the maximum of all slopes of subbundles $F\subset E$.

\noindent We remark that:
\begin{enumerate}
 \item $\mu^+(F)\leq \mu^+(E)$ for all subbundle $F\subset E$.
 \item $\mu^+(F)= \mu^+(E)$ if $F$ is maximal slope subbundle. 
 \item $\mu^+(E)= \mu(E)$  if and only if $E$ is semistable.
 
\end{enumerate}

\begin{definition}
\begin{em}
Let $(L,V)$ be a generated linear series of type $(d,r+1)$ on a curve $C$.  We say that $(L,V)$ is linearly semistable (respectively linearly stable) if for any linear subspace $W\subset V$ of dimension $w$, 
\begin{eqnarray*}
\frac{\text{deg}(L^{'})}{w-1}\geq\frac{\text{deg}(L)}{r} \hspace{.5cm}  (\text{respectively} >),
\end{eqnarray*}
where $L^{'}$ is the line bundle generated by $W$. That is, there exists the following commutative diagram 
\begin{equation*}
  \xymatrix{0  \ar[r]  & M_{W,L^{'}}   \ar[r] \ar@{^{}->}[d]      &  W \otimes\mathcal{O}_C  \ar[r] \ar@{^{}->}[d] & L^{'}   \ar[r] \ar@{->}[d]^{}  & 0\\
				0\ar[r]^{}	& M_{V,L} \ar[r]_{}& V\otimes\mathcal{O}_C  \ar[r]_{}   & L \ar[r]^{} & 0    .\\  }
				\end{equation*}
				\end{em}

\end{definition}
\noindent and the condition of being linearly (semi)stable is equivalent to the bundle $M_{V,L}$ can not be destabilized by subbundles of the form  $M_{W,L^{'}}$, where $(L^{'},W)$ is a generated subseries of $(L,V).$

 We are interested in providing conditions for the following Conjecture to be true.
 
\begin{conjecture}(\cite{mistretastopino}, Conjecture 8.7)\label{mis2}
	 Let $C$ be any curve, and let $L$ be a globally generated line bundle on $C$. The linear (semi)stability of $(H^0(C,L),L)$ is equivalent to (semi)stability for $M_L$.
\end{conjecture}

\noindent In (\cite{mistretastopino}, Conjecture 8.6), the authors also conjectured the equivalence between the stability of $M_{V,L}$ and the linear stability of the linear series $(L,V)$ when $d<\gamma r$. It is not our purpose  to study the incomplete case here, but we  give some remarks in this direction. 

\noindent Note that if $M_L$ is stable, then the above conjecture is true. Therefore, from now on we make the following  assumption: $ M_ {L} $ is strictly semistable or unstable. Let
\begin{eqnarray*}
 A:=\{S\subset M_{L} \hspace{.3cm}| \hspace{.3cm} \mu(S)\geq \mu(M_L) \}.
 \end{eqnarray*} 
We have the following remarks:
\begin{remark}\label{3} \begin{em}

\begin{itemize}

\item[(a)]If we find a bundle $S\in A$ such that the rank of the vector bundle $F_S$ that appear in Butler's diagram is one,  then the Conjecture \ref{mis2} follows.
        
        \vspace{.3cm}
        
        \item[(b)] If $h^0(L)=2$, then $M_L$ is a line bundle and the Conjecture \ref{mis2} follows.
       
       \vspace{.3cm}
       
        \item[(c)] Consider $(L,V)$  a generated linear series of type $(d,3)$ over a curve  $C$ of gonality $\gamma>\frac{d}{2}$. We remark that $M_{V,L}$ is stable bundle of rank $2$: suppose that $M_{V,L}$ is not stable, then there is a line bundle $S$ of maximal slope and $\text{deg}(S)\geq -\frac{d}{2}=\mu(M_{V,L})>-\gamma$. Since $S^{\vee}$ is a line bundle generated by $W^{\vee}$, we have that  there exists a subspace $W_{S}\subset W^{\vee}$ of dimension $2$ that generates $S^{\vee}$. The linear series $(S^{\vee},W_S)$ induces a morphism 
\begin{equation*}
\xymatrix{  \phi:C  \ar[r]^{}& \mathbb{P}^1}
 \end{equation*}
of degree $\text{deg}(S^{\vee})<\gamma$ which  contradicts the gonality of $C$. By the same argument we see that $M_{V,L}$ can not be destabilized by a line bundle $S$ when $d<\gamma r$.

               \end{itemize}
                \end{em}
\end{remark} 

We recall the following lemma that we will apply:

\begin{lemma}\label{butler}({\bf Butler}, \cite{butler}, Lemma 1.10)
Let $C$ be a curve of genus $g\geq 2$, $F$ a vector bundle on $C$ with non trivial summands such that $h^1(F)\neq 0$. Suppose that $V\subset H^0(F)$ generates $F$. If $S=M_{V,F}$ is stable, then $\mu(S)\leq -2.$ Furthermore, $\mu(S)=-2$ implies that
\begin{enumerate}
\item $C$ is a hyperelliptic curve, $F$ is the hyperelliptic bundle and $S$ its dual, or 
\item $F=K_C$ and $S=M_K$.
\end{enumerate}
\end{lemma}

\section{Butler's diagram}

In this section  we are interested in finding stability properties for vector bundles that appear in the Butler's diagram. Consider a Butler's diagram of $L$ by $S$ and denote by $Q:=M_L/S$ the quotient of $M_L$ by $S$.

\begin{remark}\label{estrictoW} 
\begin{em}
\noindent Let $C$ be a curve and $L\in\text{Pic}(C)$ a globally generated line bundle such that $h^{0}(L)=h^0(M_L^{\vee})$. We remark that $W\subsetneq H^0(L)$  if and only if $H^{0}(Q^{\vee})\neq 0$.
 To see this, note that $W\subsetneq H^0(L)$ if and only if $\phi:H^{0}(L)^{\vee}\rightarrow H^0(S^{\vee})$ is not injective. The morphism $\phi$ factorizes as follows
 \begin{equation*}
  \xymatrix{ \phi:H^{0}(L)^{\vee} \ar[r]^{\phi_1} & H^0(M_L^{\vee}) \ar[r]^{\phi_2}& H^0(S^{\vee}).\\}
	 \end{equation*}
 Since  $h^{0}(L)=h^0(M_L^{\vee})$, from the exact sequence
\begin{equation*}
 \xymatrix{ 0 \ar[r]^{}  & H^{0}(L^{\vee}) \ar[r]^{} & H^{0}(L)^{\vee} \ar[r]^{\phi_1} &  H^0(M_L^{\vee}) \ar[r]^{} & \ldots, }
\end{equation*}

\noindent we have that $\phi_1$ is an isomorphism. Thus, $\phi$ es injective if and only if $\phi_2$ is injective. But the condition  $h^{0}(Q^{\vee})\neq 0$ is equivalent that $\phi_2$ can not be injective. 

\noindent When $C$ is a general curve in the sense of Brill-Noether, the condition $h^0(L)=h^0(M_L^{\vee})$ hold (see \cite{montserrat}, Theorem 2.4 ).
\end{em}
\end{remark}

Now we prove the Theorem \ref{intromultiplicationsections}:
	
\begin{proof} We prove (1). Dualizing and twisting the Butler's diagram (\ref{butlerdiagram}) of $L$ by $S$ by the canonical line bundle $K_C$, we obtain
\begin{equation} \label{butlertensdualK}
				 \xymatrix{0  \ar[r]  & L^{\vee}\otimes K_C  \ar[r] \ar@{->}[d]      & H^0(L)^{\vee}\otimes K_C \ar[r] \ar@{->}[d] & M_L^{\vee}\otimes K_C \ar[r] \ar@{->}[d]  & 0\\
				  0  \ar[r]   &   F_S^{\vee}\otimes K_C  \ar[r]         &  W^{\vee}\otimes K_C \ar[r]    &  S^{\vee}\otimes K_C   \ar[r]           & 0      .\\  }
				\end{equation}
By the above diagram, we have 
				\begin{equation}\label{surjectivem0}
				\xymatrix{H^0(L)^{\vee}\otimes H^0(K_C) \ar[r]^{m_1} \ar@{->}[d]^{p_1}   & H^0(M_L^{\vee}\otimes K_C) \ar[r] \ar@{->}[d]^{p_2} & H^1(L^{\vee}\otimes K_C) \ar[r]^{\delta}&H^0(L)^{\vee}\otimes H^1(K_C) \ar[r] & 0 \\
				W^{\vee}\otimes H^0(K_C) \ar[r]^{m_W} & H^0(S^{\vee}\otimes K_C)     \ar@{->}[d] \\	
				  & H^1(Q^{\vee}\otimes K_C) \ar@{->}[d] \\
				    & H^1(M_L^{\vee}\otimes K_C).} 
				\end{equation}

\noindent First note that $p_1$ is surjective map because $W$ is a subspace of $H^0(L)$. The map $\delta$ is the dual of the map $H^0(L)\otimes H^0(\mathcal O_C)^{\vee}\to H^0(L)^{\vee}$ which is an isomorphism by Serre duality pairing, therefore, $\delta$ is an isomorphism and $m_1$ is surjective. Hence $m_W$ is surjective map if and only if $p_2$ is. But, by Serre duality $H^1(M_L^{\vee}\otimes K_C)=0$, the condition for surjectivity of $p_2$ is $H^1(Q^{\vee}\otimes K_C)=0$. This proves $(1)$.

\noindent To prove (2), by hypothesis and diagram (\ref{butlertensdualK}), we have the following exact sequence
	\begin{equation}\label{exactm0}
	  \xymatrix{ 0 \ar[r]^{} & H^0(F_S^{\vee}\otimes K_C) \ar[r]^{}  & W^{\vee}\otimes H^0(K_C) \ar[r]^{m_W} & H^0(S^{\vee}\otimes K_C) \ar[r]^{} & 0.\\}
	 \end{equation}

\noindent From (\ref{exactm0}) and the fact that $h^0(S)=0$, we conclude that $h^1(F_S)=(w-s)\cdot g+\text{deg}(S)+s$. By Riemann-Roch Theorem, we get
	 \begin{eqnarray*}
	  h^0(F_S)&=&f+\text{rk}(F_S)(1-g)+h^1(F_S)\\
	   &=&\text{rk}(F_S)+s\\
	   &=&w.
	 \end{eqnarray*}
\noindent Since $W\subset H^0(C,F_S)$,  we have $W=H^0(F_S)$.

\noindent We prove condition (3). First suppose that $M_L$ is semistable  and  $S\subset M_L$ is of maximal slope.  Consider  $G\subset Q$  a subbundle,  then we get 
\begin{eqnarray}\label{subbundleofQ}
\mu(G)\leq\mu(S)\leq \mu(M_{L})\leq \mu(Q),
\end{eqnarray}

\noindent this implies that $Q$ is semistable. Moreover,
\begin{eqnarray*}
 \text{deg}(Q)=\text{deg}(M_L)-\text{deg}(S)= -d +\text{deg}(F)<0.
\end{eqnarray*}
Since $Q$ is semistable of negative degree, we conclude that $H^{0}(Q)=0$.
 \vspace{.3cm}
\item Now suppose that $M_{L}$ is unstable and $S\subset M_{L}$ is of maximal slope, then we have $\mu(Q)<\mu(S)$.  Moreover, since $S\subset M_{V,L}$ stable of maximal slope,  we conclude that  $\mu(G)<\mu(S)$ for all $G\subset Q$ subbundle. Consequently, $\mu^+(Q)<\mu(S)=\mu^+(S)< 0$, and $H^{0}(Q)=0$ (see \cite{butler}, Lemma 1.12). This finishes the proof of the theorem. \end{proof}

  The Butler's diagram (1.2) is a particular case of a more general context, that is, the Lazarsfeld-Mukai bundle $M_{V,E}$ and the Butler's diagram, are defined for a globally generated coherent system $(E,V)$ of higher rank over a curve in the same way that for linear series (see for instance \cite{but}). The Theorem \ref{intromultiplicationsections} remains true for a generated coherent system $(E,H^0(E))$. Also, the properties $(1)-(4)$ of the Remark \ref{desestabilizante} are still valid, and the property (5) is as follow: Let $S\subset M_{V,E}$ be a subbundle of maximal slope, then 
 $\text{deg}(F_S)\leq \text{deg}(I)$ where $I:=\text{Im}(\alpha)$. Moreover, if $S$ is the destabilizing bundle of $M_{V,E}$, then  $\text{rk}(F_S)=\text{rk}(I)$ if and only if $\text{deg}(F_S)=\text{deg}(I)$.

\begin{remark}\label{conditionsof}
\begin{em}

 Let $L\in \text{Pic}^d(C)$ be a globally generated line bundle over a curve $C$ and set $h=h^1(L)$. Consider $S\subset M_L$ be a subbundle, we have two conditions about the vector bundle $M_L$. The first one, is to compare the slopes of $S$ and $M_L$ in terms of degree and rank of vector bundle $F_S$. The second one is proving that $h^1(F_S)$ is  bounded by the genus of the curve, the ranks of $F_S$ and $S$ when  $M_L$ is unstable and $S\subset M_L$ is of maximal slope:
 
\begin{enumerate}
\item  Suppose that $f\cdot(g-h)\leq d\cdot(\text{rk}(F_S)\cdot g-h^1(F_S))$. Applying Riemann-Roch Theorem, we have
\begin{eqnarray*}
  \mu(S)=-\frac{f}{w-\text{rk}(F_S)}   & \leq & -\frac{f}{f-\text{rk}(F_S)\cdot g +h^1(F_S)}\\
                                                &\leq&-\frac{d}{d- g +h}\\
                                                 &=&\mu(M_L).
                                                 \end{eqnarray*}
 
\noindent Moreover, we have that $\mu(S)=\mu(M_L)$ if and only if $W=H^0(F_S)$ and $f\cdot(g-h)=d\cdot(\text{rk}(F_S)\cdot g-h^1(F_S))$. In the same manner we can see that $M_{L}$ is semistable when $h^1(F_S)\leq (\text{rk}(F_S)-1)\cdot g +h$ and $S\subset M_{L}$ is a  subbundle of maximal slope, this happens because $f\leq d$. Hence, from now on we make the assumption 
	\begin{eqnarray}\label{cotah1}
	h^1(F_S)>(\text{rk}(F_S)-1)\cdot g +h.
	\end{eqnarray}
Suppose now that $M_L$ is unstable and $S\subset M_L$ is of slope maximal. By Theorem \ref{intromultiplicationsections}, $W=H^0(F_S)$. We have 
\begin{eqnarray}\label{desigualdades}
 \frac{-f}{h^0(F_S)-\text{rk}(F_S)} =\frac{-f}{w-\text{rk}(F_S)}=\mu(S)\geq\mu(M_L)=\frac{-d}{d-g+h}.
\end{eqnarray}
From  \ref{desigualdades}, we obtain
\begin{eqnarray*}
 f\cdot(g-h)\geq d\cdot(\text{rk}(F_S)\cdot g-h^1(F_S)).
\end{eqnarray*}
Hence  the above inequality is a relationship between tha slopes of $S$ and $M_L$ in terms of the degree and rank of $F_S$.
\item
Suppose that $M_L$ is unstable and $S\subset M_L$ is stable of maximal slope, then $W=H^0(F_S)$ and  we conclude that
\begin{eqnarray*}
\text{rk}(F_S)\cdot g- h^1(F_S)=f-s.
\end{eqnarray*}
If $h^1(F_S)> \text{rk}(F_S)\cdot g-s$, then $f-s< s$ which is equivalent to $\mu(S)> -2,$ but this is impossible by Lemma \ref{butler}. Moreover, if $h^1(F_S)= \text{rk}(F_S)\cdot g-s$, then $\mu(S)=-2$ and $F_S$ is the canonical line bundle. Therefore 
\begin{eqnarray}\label{cotaporarr}
  h^1(F_S)\leq \text{rk}(F_S)\cdot g-s.
\end{eqnarray}
Hence $h^1(F_S)$ is bounded by the genus of the curve, the ranks of $F_S$ and $S$. 
\end{enumerate}
\end{em}
\end{remark}
	
\noindent We denote by $\chi(E)$ the Euler characteristic of a vector bundle $E$ over a curve $C$. In the following lemma we give conditions which implies that the rank of $F_S$ is one.

 \begin{lemma} Let $L\in \text{Pic}^d(C)$ be a globally generated line bundle over $C$. Consider $S\subset M_L$ of maximal slope, then  Conjecture \ref{mis2} is true  if one of the following conditions hold
	\begin{enumerate}
	 \item[(i)] $\chi(F_S)\geq 0.$
	 \item[(ii)]$h^1(F_S)< g.$
	\end{enumerate}\end{lemma}

	\begin{proof} If $\text{deg}(L)>2g$ then $M_L$ is stable and the Conjecture \ref{mis2} follows. In (\cite{mis}, Lemma 2.2), E.C. Mistretta proved the linear stability of $L$ is equivalent to the stability of $M_L$ when $d= 2g$.  By Remark \ref{desestabilizante} we have that $f=\text{rk}(F_S)\leq \text{deg}(L)=d$ and $f=d$ then the rank of $F_S$ is one and this implies that Conjecture \ref{mis2} is true. Hence suppose that $f<d\leq 2g-1$: 
	 \begin{itemize}
	  \item[(i)]$\chi(F_S)\geq 0 \Longleftrightarrow \mu(F_S)\geq g-1$. If $\text{rk}(F_S)\geq 2$, then
	  \begin{eqnarray*}
	   \mu(F_S)= \frac{f}{\text{rk}(F_S)}<\frac{d}{\text{rk}(F_S)}\leq \frac{2g-1}{2}\leq g,
	  \end{eqnarray*} 
	  which is imposible.

	\item[(ii)] From Remark \ref{conditionsof}(1) we have that $(\text{rk}(F_S)-1)\cdot g +h < h^1(F_S)< g$. This implies that $\text{rk}(F_S)< 2-\frac{h}{g}$. Hence $\text{rk}(F_S)=1$ and this completes the proof.
\end{itemize}
\end{proof}
\section{Linear stability and stability on general curves}
 In this section we prove  Conjecture $\ref{mis2}$ when $C$ is a  general curve, also we give a 
 proof on a missing gap on the conditions for the stability of $M_L$ when  $C$ is  general curve (see \cite{but}, Theorem 2.2). 

\begin{lemma}\label{semistabilityML}
	Let $C$ be a general curve of genus $g\geq 2$, and  let $L\in \text{Pic}^d(C)$ be a  globally generated line bundle with $h^0(L)=r+1$. We have that  $M_L$ is semistable, moreover, if there exists a subbundle $S\subset M_L$ such that $\mu(S)=\mu(M_L$)  then 
\begin{enumerate}
 \item $h^1(L)=0$.
 \item  $s=r-1$.
 \item $d=g+r$ and $r$ divides $g$.
\end{enumerate}
\end{lemma}

\begin{proof}
 Consider a stable subbundle  $S\subset M_L$ of maximal slope. Since $S^{\vee}$ is a quotient of $M_{L}^{\vee}$ and $M^{\vee}_L$ is a globally generated vector bundle, it follows that $S^{\vee}$ is generated by global sections. Hence for a general subspace $U\in \text{Grass}(s+1, H^0(S^{\vee}))$, there exists a short exact sequence
\begin{equation*}
 \xymatrix{0\ar[r]^{} & S  \ar[r]^{}& U^{\vee}\otimes \mathcal{O}_C \ar[r]^{} & \text{det}(S^{\vee}) \ar[r]^{} &  0,}
\end{equation*}
which induces the exact sequence in cohomology
\begin{equation*}
 \xymatrix{0\ar[r]^{} &  H^0(S) \ar[r]^{}& U^{\vee} \ar[r]^{} &  H^0(\text{det}(S^{\vee})) \ar[r]^{} &  \ldots,}
\end{equation*}
The bundle $S$ is semistable of negative degree, then $h^{0}(S)=0$ and $h^0(\text{det}(S^{\vee}))\geq \text{dim}(U)=s+1$. Since $C$ is a general curve and $\text{det}(S^{\vee})$ is a line bundle of degree  $\text{deg}(S^{\vee})$ with at least  $s+1$ sections, we have that the Brill-Noether number for 
$\text{det}(S^{\vee})$ is
\begin{equation}\label{brillnoethernumber}
\rho(g,\text{deg}(S^{\vee}),s+1)= g-(s+1)(g-\text{deg}(S^{\vee}) + s)\geq 0
\end{equation}
 By \ref{brillnoethernumber} we have 
 \begin{equation*}
 \text{deg}(S^{\vee})\geq \frac{s\cdot(g+s+1)}{s+1},
 \end{equation*}
which is equivalent to
 \begin{equation}
 \mu(S)\leq -1-\frac{g}{s+1}.
 \end{equation}
Set $h=h^1(L)$. By Riemann-Roch Theorem we have 
\begin{equation}
\mu(M_L)=-\frac{d}{r}=-1+\frac{h-g}{r},
\end{equation}
 therefore
 \begin{equation}\label{relacionslopeSM}
 \mu(S)- \mu(M_L)\leq   g\cdot  (\frac{1}{r}-\frac{1}{s+1})-\frac{h}{r}\leq 0,
 \end{equation}
 hence $M_L$ is semistable. Using \ref{relacionslopeSM}, we have that $h=0$ and $s=r-1$ when $\mu(S)=\mu(M_L)$. Applying Riemann-Roch Theorem we have that $d=g+r.$ Finally, $r$ divides $g$ because $\mu(M_L)=\text{deg}(M_L/S)\in \mathbb{Z}$, which completes the proof.
\end{proof}
We are interested in studying the Butler's diagram of $L$ by $S$ when $M_L$ is strictly semistable and $S\subset M_L$ is of maximal slope. We have the following

\begin{theorem}\label{generalcomplete}
	Let $C$ be a general curve of genus $g\geq 2$ and  let $L\in \text{Pic}^d(C)$ be a  globally generated line bundle with $h^0(L)=r+1$. Consider $S\subset M_L$ a subbundle with $\mu(S)=\mu(M_L)$, then there exists a line bundle $F=F_S$ which fits into the following commutative diagram
\begin{equation*} 
				 \xymatrix{0  \ar[r]  & S  \ar[r] \ar@{->}[d]      & H^0(F)\otimes\mathcal{O}_C \ar[r] \ar@{->}[d] & F \ar[r] \ar@{->}[d]  & 0\\
				  0  \ar[r]   &   M_L  \ar[r]         &  H^0(L)\otimes \mathcal{O}_C \ar[r]    &  L   \ar[r]           & 0      .\\  }
				\end{equation*}

				\end{theorem}

\begin{proof} Consider the Butler's diagram (\ref{butlerdiagram}) of $L$ by $S$ where $S\subset M_L$ is a subbundle of maximal slope such that $\mu(S)=\mu(M_L)$. The fact that $W=H^0(F)$ follows from the surjectivity of $m_W:W^{\vee}\otimes H^{0}(K_C)\rightarrow H^0(S^{\vee}\otimes K_C)$ (see Theorem \ref{intromultiplicationsections}).  From Lemma \ref{semistabilityML} we have that $h=h^1(L)=0$ and $s=r-1$ when $\mu(S)=\mu(M_L)$. Since $W$ is a subspace of $H^0(L)$ and the dimension of $W$ is greater than the rank of $S$, it follows that  the rank of $F_S$ is $1$ or $2$. We will prove that the rank of $F_S$ is 1.
 
 \vspace{.2cm}
 
\noindent{\it{Claim.}} $h^0(\text{det}(F_S))=r.$ 
\vskip1mm
\noindent{\it{Proof of Claim.}} First, note that $\text{deg}(F)=g+r-1-\frac{g}{r}$. This follows from the fact that $\mu(S)=\mu(M_L)$ and $\text{deg}(F_S)=-\text{deg}(S)$. Note also that
 \begin{eqnarray*}
  h^0(\text{det}(F_S))=h^0(\text{det}(S^{\vee}))\geq s+1=r.
 \end{eqnarray*}
Suppose that $h^0(\text{det}(F_S))\geq r+1$. Since $C$ is a general curve and $\text{det}(F_S)$ is a line bundle of degree $\text{deg}(F_S)$ with at least $r+1$ sections, we conclude that the corresponding Brill-Noether number $\rho$ for the line bundle  $\text{det}(F_S)$ is nonnegative, but $\rho=g-(r+1)\cdot (r-\text{deg}(F)+g)=g-(r+1)\cdot (g/r+1) <0$, which contradicts that $C$ is general. This proves the claim. 
\vskip2mm

\noindent Now, we need to distinguish two cases: $r=2$ and $r\neq 2$.

\vspace{.2cm}
 
\noindent{\it Case 1.} Consider the case $r=2$. Therefore $d=g+2$ and $h=0$. Since $M_L$ is strictly semistable, there exists a unique line bundle $S\subset M_L$ of maximal degree 
$-\frac{g+2}{2}$. By the claim, $h^0(S^{\vee})=2$.  We recall that $W^{\vee}:=\text{Im}(H^0(L)^{\vee}\stackrel{\phi}{\rightarrow}H^0(S^{\vee}))$  and generates $S^{\vee}$(see diagram  \ref{butlerdiagram}). We see that $\phi$ is surjective, $F_S=S^{\vee}$ is a line bundle and $W=H^0(S^{\vee})$ has dimension $2$. Hence the Butler's diagram of $L$ by $S$ is giving by
\begin{equation*} 
				 \xymatrix{0  \ar[r]  & S  \ar[r] \ar@{->}[d]      & H^0(S^{\vee})\otimes\mathcal{O}_C \ar[r] \ar@{->}[d] & S^{\vee} \ar[r] \ar@{->}[d]  & 0\\
				  0  \ar[r]   &   M_L  \ar[r]         &  H^0(L)\otimes \mathcal{O}_C \ar[r]    &  L   \ar[r]           & 0      .\\  }
				\end{equation*}
				
				which complete the case $r=2$.
				
				\vspace{.1cm}
				
\noindent{\it Case 2.}  $r\neq 2$. By hypothesis the gonality of $C$ is $\gamma\geq \frac{g+2}{2}$, which implies that 
\begin{equation*}
 d=g+r < \frac{g+2}{2}\cdot r\leq \gamma \cdot r.
\end{equation*}
Suppose that $\text{rk}(F_S)=2$. By Theorem \ref{intromultiplicationsections} and the Butler's diagram we have that  $h^0(F_S)=\text{dim}(W)=r+1$. Since $h^0(\text{det}(F_S))=r$, it follows that  the short exact sequence 
\begin{equation*}
 \xymatrix{0\ar[r]^{} & \mathcal{O}_C   \ar[r]^{}& F_S \ar[r]^{} & \text{det}(F_S) \ar[r]^{} &  0,}
\end{equation*}
is exact on global sections. This implies that $\mu(S)<\mu(M_L)$ (see \cite{mistretastopino}, Lemma 4.3), which is impossible.  Thus, $F=F_S$ is a line bundle. This completes the proof.
\end{proof}

\noindent We recall that the advantage in using that the rank of $F=F_S$ is one lies in the fact  that linearly stable is equivalent to stability of $M_L$. It is worth pointing out that the theorem above gives as a corollary  the stability conditions for $M_L$ stated by Butler:

\begin{corollary}\label{equivalenciageneral} Under the hypothesis of Theorem \ref{generalcomplete} we have:
\begin{enumerate}
	\item  The stability of $M_L$ is equivalent the linear stability of $(L, H^0(L))$.
	
	\vspace{.1cm}
	
	\item $M_L$ fails to be stable if and only if all the following three condition hold
	
	\vspace{.1cm}
	
	 \begin{enumerate}
	  \item $h^1(L)=0$.
	  
	  \vspace{.1cm}
	  
	  \item $d=g+r$ and $r$ divides  $g$.
	  
	  \vspace{.1cm}
	  
	  \item There is an effective divisor $Z$ with $h^0(L(-Z))=h^0(L)-1$ and $\text{deg}(Z)=1+\frac{g}{r}$.
	 \end{enumerate}
	 \end{enumerate}
	\end{corollary}
\begin{proof} If $M_L$ is stable, then $(L, H^0(C,L))$ is linearly stable. Suppose now that $M_L$ is strictly semistable and let $S\subset M_L$ a subbundle such that $\mu(S)=\mu(M_L)$. By Theorem \ref{generalcomplete} the rank of $F_S$ is one and this implies that $(L,H^0(L))$ is strictly linearly semistable. This proves (1).
	
	\vspace{.2cm}
	
	\item We prove (2). ($\Leftarrow )$ Note that the conditions $(a)$, $(b)$ and $(c)$ implies that $F_S=L(-Z)$ in the Butler diagram of $L$ by $S:=M_{L(-Z)}$ and  $\mu(M_{L(-Z)})=\mu(M_L)$. Hence $M_L$ is semistable but not stable. 

\vskip1mm
	
\noindent ($\Rightarrow)$ Now, suppose that $M_L$ is strictly semistable. From Lemma $\ref{semistabilityML}$ there exists a subbundle $S\subset M_L$ such that the conditions $(a)$ and $(b)$ are satisfied. Since $F_S$ is a line bundle and $\alpha :F_S\rightarrow L$ is not a zero map, we have that there exists an effective divisor $Z$ such that $F=L(-Z)$ (see Lemma \ref{generalcomplete}). From the fact that $\text{deg}(F)=-\text{deg}(S)$ and $\mu(S)=\mu(M_L)$ it follows that $\text{deg}(Z)=1+g/r$. Note that $Q:=M_L/S=\mathcal{O}(-Z)$, then $h^0(Q)=0$, so $m_W$ is surjective by Theorem 1.1, this implies that $h^0(L(-Z))=h^0(F)=w=h^0(L)-1$, this gives condition $(c)$.\end{proof}

	\begin{example}
	 \begin{em}
	 Let $C$ be a general curve of genus even and let $D$ be an effective divisor such that $|\mathcal O_C(D)|=g_{\frac{g+2}{2}}^1$ is free of base points. Take $L=\mathcal{O}_C(2D)$, therefore $L$ satisfies the properties (a), (b) and the divisor $Z$ is precisely $D$, this gives the condition $(c)$. Hence $M_L$ is semistable but not stable.
	 \end{em}

	\end{example}

\section{$\gamma$-Gonal curves}

In this section we consider a curve $C$ of gonality $\gamma$. We denote by $B$ a line bundle (or $\gamma$-gonal divisor) that computes the gonality, i.e 
$\text{deg}(B)=\gamma$ and $h^0(C,B)=2$. The relationship between the gonality  and the index Clifford  of a curve  is given by the following inequality 
\begin{eqnarray*}
 \gamma -3\leq \text{Cliff}(C)\leq \gamma-2.
\end{eqnarray*}

\noindent In (\cite{mistretastopino}, Theorem 1.1), the authors proved the equivalence between the stability of $M_L$ and the linear stability of $(L,H^0(L))$ when 
$ d-2\cdot r\leq\text{Cliff}(C)$. This condition implies that 
	 \begin{eqnarray*}
	  d\leq 2\cdot r+\text{Cliff}(C) \leq  2\cdot r +\gamma-2\leq \gamma\cdot r.
	 \end{eqnarray*}

\noindent So, we are interested  in the case $d>\gamma\cdot  r$. We have the following criterium  we solve a particular case of the  Conjecture $\ref{mis2}$.
	 
	 \begin{theorem}\label{gonalinyectivo}Let $C$ be a curve of gonality $\gamma$, and let $L\in \text{Pic}^d(C)$ be a globally generated line bundle with $h^0(C,L)=r+1$. Suppose that $d > \gamma \cdot  r$.
	\noindent If the multiplication map 
 \begin{equation*}
  \xymatrix{ \mu_{L}^{\gamma}:=H^0(L)\otimes H^0(C,B) \ar[r]^{}& H^0(L\otimes B)}
 \end{equation*}

 is not injective, then $(L, H^0(L))$ is not linearly semistable.
	 \end{theorem}

	 \begin{proof}
	 Twisting the short sequence (\ref{dualspam}) by $B$, we get
\begin{equation*}
 \xymatrix{0\ar[r]^{} & M_L\otimes B   \ar[r]^{}&  H^0(L)\otimes B \ar[r]^{} & L\otimes B \ar[r]^{} &  0.}
\end{equation*}
This sequence induce the following sequence in cohomology
\begin{equation*}
 \xymatrix{0\ar[r]^{} & H^{0}(M_L\otimes B)   \ar[r]^{}& H^0(L)\otimes H^0(B) \ar[r]^{\mu_{L}^{\gamma}} & H^{0}(L\otimes B) \ar[r]^{} &  \ldots}
\end{equation*}

Note that the multiplication map $\mu^{\gamma}_{L}$ is not injective if and only if $h^0(M_L\otimes B)\neq 0$.  By the base point free pencil trick, $\mu^{\gamma}_L$ is not injective  if and only if $H^0(B^{\vee}\otimes L)\neq 0$. If $h^0(B^{\vee}\otimes L)>0$, therefore $B\subset L$ and we have the following diagram 
\begin{equation} \label{butlerdiagramgonal} 
				 \xymatrix{0  \ar[r]  &  B^{\vee}   \ar[r] \ar@{->}[d]      & H^0(B)\otimes\mathcal{O}_C \ar[r] \ar@{->}[d] & B \ar[r] \ar@{->}[d]  & 0\\
				  0  \ar[r]   &   M_L  \ar[r]         &  H^0(L)\otimes \mathcal{O}_C \ar[r]    &  L   \ar[r]           & 0      .\\  }
				\end{equation}
				
In particular, we get 
\begin{eqnarray*}
	  \mu(B^{\vee})=-\gamma > -\frac{d}{r}=\mu(M_L)
	 \end{eqnarray*}
	 It follows that $(L,H^0(L))$ is not linearly semistable, which completes the proof.
\end{proof}

\begin{remark}
 \begin{em}
 If we take $S=B^{\vee}$ in the Butler's diagram, we obtain 
\begin{equation*}
 \xymatrix{ W^{\vee}:=\text{Im}(H^0(L)^{\vee}\ar[r]^{\phi} & H^0(C,B)).}
\end{equation*}

Since $B$ is a generated line bundle, we see that $\phi$ is surjective and $W$ has dimension 2. The Butler's diagram of $L$ by $S$ is the diagram \ref{butlerdiagramgonal}.
 \end{em}
\end{remark}

\noindent The principal significance of the above Theorem is that $M_L$ is not semistable. The following lemma gives a condition for  which the multiplication map $\mu_{L}^{\gamma}$ is not injective.

	 \begin{lemma}\label{seccionesconjecture}
	   Under the hypothesis of Theorem \ref{gonalinyectivo}. If $d>\gamma+g-1-h^1(L)$, then the Conjecture \ref{mis2} follows.
	  \end{lemma}

	 \begin{proof}
	From Theorem \ref{gonalinyectivo} we only need to show that, under the above hypothesis, the map $\mu_L^{\gamma}$ is not injective.

\noindent First note that $h^1(L)\geq h^1(L\otimes B)$. To see this, let $D\in |B|=g_{\gamma}^1$ be a gonal-divisor. The exact sequence
\begin{equation*}
 \xymatrix{0\ar[r]^{} & L  \ar[r]^{}& L\otimes \mathcal O_C(D) \ar[r]^{} &  \mathcal{O}_{D}\otimes L\otimes \mathcal O_C(D)\ \ar[r]^{} &  \ldots,}
\end{equation*}
induces an exact sequence
\begin{equation*}
 \xymatrix{ \cdots \ar[r]^{} & H^0( \mathcal{O}_{D}\otimes L\otimes B) \ar[r]^{}&  H^1(L) \ar[r]^{} & H^1(L\otimes B)  \ar[r]^{} & 0,}
\end{equation*}

therefore $h^1(L)\geq h^1(L\otimes B)$.
 
\noindent Counting dimensions and using that $h^1(L)\geq h^1(L\otimes B)$ and applying the hypothesis on the degree $d$, we see that the map $ \mu_{L}^{\gamma}$ can not be injective:
\begin{eqnarray*}
h^0(C, B)\cdot h^0(C, L)&=& 2\cdot h^0(L)\\
                 &=& 2d+2-2g+2h^1(L)\\
                 &\geq& 2d+h^1(L)+2-2g +h^1(C, L\otimes B)\\
                 &>& d+\gamma+1-g + h^1C, (L\otimes B)\\
                 &=&h^0(C, L\otimes B).
\end{eqnarray*}
This proves the Lemma.
	 \end{proof}

	 \begin{example}
	  \begin{em}
	 Let $C$ be a curve of gonality $\gamma$ and let $L$ be a generated line bundle with $h^1(L)=0$. The conditions $d>\gamma \cdot (h^0(L)-1)$ and   $d+h^1(L)>\gamma+g-1$ hold if  $g>(\gamma -1)^2$.
	 \end{em}
	 \end{example}

\begin{example}
\begin{em} Let $C$ be a non-hyperelliptic curve and $L\in \text{Pic}^{2g-1}(C)$ be a globally generated line bundle, then $\text{deg}(L)-2\cdot (h^0(L)-1)\leq \text{Cliff}(C)$ and Conjecture \ref{mis2} is satisfied (see \cite{mistretastopino}, Theorem 1.1). 

\vspace{.2cm}
	  
\noindent Let $C$ be an hyperelliptic curve of genus $g$. If $g=2$ then  $h^0(L)=2$ and by Remark \ref{3} the Conjecture \ref{mis2} follows. If  $g\geq 3$, then $h^0(L)=g\geq 3$ and  by Lemma \ref{seccionesconjecture} the Conjecture \ref{mis2} is satisfied.
	 \end{em}

	\end{example}

	 \begin{corollary}
	  The conjecture \ref{mis2} is true for hyperelliptic curves.
	 \end{corollary}
	 \begin{proof}
	If $d\leq 2r$, then $L$ satisfies M-S (see \cite{mistretastopino}, Theorem 1.1). The case $d>2r$ follows from  Lemma \ref{seccionesconjecture} and Remark \ref{3} (b).  
	
	 \end{proof}

\noindent{\bf Acknowledgments:} Second named author warmly thanks the Centro de Ciencias en Matem\'aticas (UNAM) in Morelia City for their hospitality and the use of their resources during a posdoctoral year. Both authors were supported by Research Grant PAPIIT IN100716 (UNAM). The second named author is supported by a Posdoctoral Fellowship from CONACyT, M\'exico.


%

\begin{thebibliography}{999}
\bibitem{bra}

{ L. Brambila-Paz}: Non-emptiness of moduli space of coherent systems.
 International Journal of Math. {\bf 19}, no. 7, 777--799 (2008).




\bibitem{butler}
{ Butler, D. C.}: Normal generation of vector bundles over a curve.
 J. Differential Geom. {\bf 39}, no. 1, 1--34 (1994).

 \vspace{.1cm}

\bibitem{but}
{ Butler, D. C.}:
 Birational maps of
moduli of Brill-Noether pairs. Preprint, arXiv:alg-geom/9705009.

 \vspace{.1cm}

\bibitem{bbn1}
	U.N. Bhosle, L. Brambila-Paz and P.E. Newstead:
	{\it On coherent systems of type $(n, d, n + 1)$ on Petri curves},
	Manuscripta Math. {\bf 126},  (2008), 409--441.

 \vspace{.1cm}
	
\bibitem{ein}
{ Ein, L.,  Lazarsfeld, R.}:{\it Stability and restrictions of Picard bundles, with an application to the normal bundles of elliptic curves}. In: Ellingsrud, G., Peskine, C., Sacchiero, G.,
Stromme, S.A. (eds.) Complex Projective Geometry (Trieste 1989/Bergen 1989). LMS
  Lecture Note Series, vol. 179, pp. 149--156. CUP, Cambridge (1992).

   \vspace{.1cm}


\bibitem{resolucionminimal} G. Farkas, M. Mustata and M. Popa: {\it Divisors on $\mathcal{M}_{g,g+1}$ and the minimal resolution conjecture for points on canonical curves},  Annales Sci. de \'Ecole Norm. Sup. (4)  {\bf 36}, (2003), 553--581.

 \vspace{.1cm}


\bibitem{mis}
{Mistretta,  E. C.}: {\it Stability of line bundle transforms on curves with respect to low codimensional subspaces}. J. Lond. Math. Soc. (2) {\bf 78}, no. 1, 172--182 (2008).

 \vspace{.1cm}

%
\bibitem{mistretastopino}
{ Mistretta, E. C.; Stoppino, L.}: {\it Linear series on curves: stability and Clifford index}.
  Internat. J. Math. {\bf 23}, no. 12 (2012).

 \vspace{.1cm}
  
\bibitem{mumfordstability}
{ Mumford, D.}: {\it Stability of projective varieties}. Lectures given at the "Institut des Hautes \'Etudes Scientifiques", Bures-sur-Yvette, March-April 1976. Monographie de l'Enseignement Math\'ematique, No. 24. L'Enseignement Math\'ematique, Geneva, (1977). 74 pp.
  
   \vspace{.1cm}

\bibitem{ram}
{ Paranjape, K., Ramanan, S.}: {\it On the canonical ring of a curve. In: Algebraic geometry and Commutative Algebra}, in Honor of Masayoshi Nagata, vol. 2, pp. 503--516 Kinokuniya (1987).

 \vspace{.1cm}

\bibitem{schneider} Schneider, O.: {\it Stabilit\'e des fibr\'es  $\wedge^p E_L$ et condition de Raynaud }. Fac. Sci. Toulouse Math. (6) 14(3), 515-525, (2005). 

 \vspace{.1cm}

\bibitem{montserrat} Gwena, Tawanda; Teixidor i Bigas: {\it Montserrat Maps between moduli spaces of vector bundles and the base locus of the theta divisor}. Proc. Amer. Math. Soc. 137 (2009), no. 3, 853-861.

 \vspace{.1cm}

	

\end{thebibliography}
\end{document}